\documentclass[11pt]{amsart}
\newcommand\version{July 28, 2026}

\usepackage[T1]{fontenc}
\usepackage{lmodern}
\usepackage{microtype}
\usepackage{mathtools}
\usepackage{amssymb}
\usepackage{enumitem}
\usepackage[colorlinks=true,linkcolor=blue,citecolor=blue,urlcolor=blue]{hyperref}
\usepackage[nameinlink,noabbrev]{cleveref}
\usepackage[margin=1.12in]{geometry}

\newtheorem{theorem}{Theorem}[section]
\newtheorem{proposition}[theorem]{Proposition}
\newtheorem{lemma}[theorem]{Lemma}
\newtheorem{corollary}[theorem]{Corollary}

\theoremstyle{definition}

\theoremstyle{remark}
\newtheorem{remark}[theorem]{Remark}

\newcommand{\T}{\mathbb T}
\newcommand{\Sone}{\mathbb S^1}
\newcommand{\Z}{\mathbb Z}
\newcommand{\R}{\mathbb R}
\newcommand{\N}{\mathbb N}
\newcommand{\C}{\mathbb C}

\newcommand{\wh}{\widehat}

\DeclareMathOperator{\degmap}{deg}

\title[Winding number at the critical H\"older endpoint]
{The Winding Number at the Critical H\"older Exponent $1/3$:\\
Failure of Universal Fourier Summation}

\hypersetup{
  pdftitle={No Universal Fourier Summation Formula for the Degree of Circle Maps},
  pdfauthor={Rupert L. Frank and Paata Ivanisvili},
  pdfsubject={Fourier summation processes and the winding number},
  pdfkeywords={winding number, topological degree, Fourier coefficients, summation process, fractional Sobolev space, Hölder space}
}

\author{Rupert L. Frank}
\address[Rupert L. Frank]{Mathematisches Institut, Ludwig-Maximilians Universit\"at M\"unchen, Theresienstr.~39, 80333 M\"unchen, Germany; Munich Center for Quantum Science and Technology, Schellingstr.~4, 80799 M\"unchen, Germany; and Mathematics 253-37, Caltech, Pasadena, CA 91125, USA}
\email{r.frank@lmu.de}

\author{Paata Ivanisvili}
\address[Paata Ivanisvili]{Department of Mathematics, University of California, Irvine, 510C Rowland Hall, Irvine, CA 92697-3875, USA}
\email{pivanisv@uci.edu}

\date{\version}

\begin{document}

\begin{abstract}
The degree (or winding number) of a sufficiently regular map $f:\T\to\Sone$ is given in terms of its Fourier coefficients by
\[
 \degmap f=\sum_{n\in\Z}n\,|\wh f(n)|^2 \,.
\]
At lower regularity the series may diverge, but, as shown by Kahane, when $f$ is $\alpha$-H\"older continuous with $\alpha>1/3$, then the degree can be recovered by a universal linear summation process. We show that no summation process satisfying Brezis's natural axioms can recover the degree universally for $\alpha=1/3$, thereby resolving an open problem by Brezis.
\end{abstract}

\maketitle


\section{Introduction}

We are interested in maps $f:\T\to\Sone$ and their degree. Here, for the sake of clarity, we distinguish between the domain $\T=\R/(2\pi\Z)$ and the target $\Sone=\{ z\in\C:\ |z|=1\}$. 

The degree can be defined for maps $f\in VMO(\T;\Sone)$, as shown by Brezis and Nirenberg \cite{BrezisNirenberg1995}, but for most of the present paper the degree for continuous maps $f$ suffices. In this case, we can write $f=e^{i\varphi}$ with a continuous function $\varphi:\R\to\R$ satisfying
$$
\varphi(x+2\pi)=\varphi(x)+2\pi d
\qquad\text{for all}\ x\in\R
$$
with some $d\in\Z$, and one has $d=\deg f$. For more background we refer to \cite{Brezis2006,BrezisMironescu2021}.

The starting point of our paper is a formula for the degree in terms of the Fourier coefficients of $f$, namely,
\begin{equation}\label{eq:smooth-degree}
 \degmap f =\sum_{n\in\Z}n\,|\wh f(n)|^2 
 \qquad\text{for all}\ f\in W^{1/2,2}(\T;\Sone) \,,
\end{equation}
where the Fourier coefficients are defined by
$$
\wh f(n)=\int_\T f(x)e^{-inx}\,\frac{dx}{2\pi},\qquad n\in\Z \,.
$$
For smooth maps $f$, equation \eqref{eq:smooth-degree} follows easily from $\degmap f = (2\pi i)^{-1} \int_0^{2\pi} \overline{f(x)}f'(x)\,dx$ and Parseval's formula. It is not difficult to extend the formula, with absolute convergence, to maps $f$ in the Sobolev space $W^{1/2,2}(\T;\Sone)$; see \cite{Brezis1997}.

Identity \eqref{eq:smooth-degree} becomes subtle for regularity below $W^{1/2,2}$. Natural scales of spaces where this question has been investigated are the H\"older spaces $C^{0,\alpha}$ and the Sobolev spaces $W^{1/p,p}$, $1<p<\infty$. For the latter spaces the degree is well defined since $W^{1/p,p}\subset VMO$ as recalled, for example, in \cite{Brezis2006}.

It turns out that the critical regularity is at $s=1/3$ derivatives. Indeed, Kahane \cite{Kahane2005} and Brezis \cite{Brezis2006} showed that
\begin{equation}\label{eq:sinc-formula}
 \degmap f
 =\lim_{\varepsilon\downarrow0}
 \sum_{n\ne0}n\,|\wh f(n)|^2\frac{\sin(n\varepsilon)}{n\varepsilon}
 \qquad \text{for all}\ f\in W^{1/3,3}(\T;\Sone) \,.
\end{equation}
Conversely, Kahane proved \cite[Proposition~P4]{Kahane2005} that
\eqref{eq:sinc-formula} may fail for certain
$f\in C^{0,1/3}(\T;\Sone)$. This failure even holds in the strong sense
that for any $\lambda\in\R$ there is an
$f\in C^{0,1/3}(\T;\Sone)$ with $\degmap f=0$ such that the regularized
expression on the right side of \eqref{eq:sinc-formula} converges to
$\lambda$.

These results do not rule out the possibility of a regularization procedure, different from the one in \eqref{eq:sinc-formula}, for which the analogue of this formula remains valid below regularity $1/3$.

Our main result shows that this is not the case: there is no summation process satisfying the natural assumptions proposed by Brezis for which the formula \eqref{eq:smooth-degree} can be regularized at H\"older regularity $1/3$ or below Sobolev regularity $1/3$. This gives a negative answer to one of Brezis's favorite open problems \cite[Open Problem~5.6]{Brezis2023}.

\begin{theorem}[Main theorem]\label{thm:main}
Let $(\sigma_{n,\varepsilon})_{n\in\Z,\,0<\varepsilon<1}\subset\C$ be a family such that
\begin{align}
 &\sup_{n\in\Z}|n|\,|\sigma_{n,\varepsilon}|<\infty
 &&\text{for every fixed }\varepsilon\in(0,1),\label{eq:process-bounded}\\
 &\sigma_{n,\varepsilon}\longrightarrow1
 &&\text{as }\varepsilon\downarrow0,\text{ for every fixed }n\in\Z.\label{eq:process-pointwise}
\end{align}
Then there is a map
\[
 f\in C^{0,1/3}(\T;\Sone)
\]
for which
\begin{equation}
    \label{eq:main}
     \sum_{n\in\Z}n\,|\wh f(n)|^2\sigma_{n,\varepsilon}
 \not\longrightarrow\degmap f
 \qquad\text{as }\varepsilon\downarrow0.
\end{equation}
\end{theorem}

\begin{remark}
    (a) Assumptions \eqref{eq:process-bounded} and \eqref{eq:process-pointwise} on the summation process are precisely the assumptions in \cite[Eqs.~(5.11) and (5.12)]{Brezis2023}. Assumption \eqref{eq:process-bounded} makes $\sum_{n\in\Z}n\,|\wh f(n)|^2\sigma_{n,\varepsilon}$ absolutely convergent for $\varepsilon\in(0,1)$. Assumption \eqref{eq:process-pointwise} is in fact forced, for every $n\ne0$, by universal recovery of the degree: applying the process to $f(x)=e^{inx}$ gives $n\sigma_{n,\varepsilon}\to n$. The value of $\sigma_{0,\varepsilon}$ is irrelevant.\\
    (b) Theorem \ref{thm:main} gives a negative answer to both the H\"older and the Sobolev parts of \cite[Open Problem~5.6]{Brezis2023}. This follows from the inclusions $C^{0,1/3}(\T;\Sone)\subset C^{0,\alpha}(\T;\Sone)$ for every $0<\alpha\leq 1/3$ and $C^{0,1/3}(\T;\Sone)\subset W^{1/p,p}(\T;\Sone)$ for every $p>3$. The latter inclusion is a consequence of
    $$
    [f]_{W^{1/p,p}}^p = \iint_{\T\times\T} \frac{|f(x)-f(y)|^p}{d_\T(x,y)^2}\,dx\,dy \leq C_p \sup_{0<|h|\leq\pi} \frac{\|f(\cdot+h)-f\|_{L^\infty}^p}{|h|^{p/3}} = C_p [f]_{C^{0,1/3}}^p \,,
    $$
    where
    $$
    C_p := 2\pi\int_{-\pi}^\pi |h|^{p/3-2}\,dh<\infty \qquad\text{when }p>3 \,.
    $$
    Here and in what follows use the norm and seminorm
    \[
    \|u\|_{C^{0,\alpha}}
    :=\|u\|_{L^\infty}
    +[u]_{C^{0,\alpha}},
    \qquad
    [u]_{C^{0,\alpha}}
    :=\sup_{x\ne y}
    \frac{|u(x)-u(y)|}{d_\T(x,y)^\alpha}.
    \]
    (c) A weaker form of the main theorem, with $f\in C^{0,1/3}$ replaced by $f\in C^{0,\alpha}$ for any given $\alpha<1/3$ was recently obtained in \cite{Cieszynski2026}. This suffices to provide a negative answer to the part of \cite[Open Problem~5.6]{Brezis2023} that concerns the Sobolev scale. To provide a negative answer to the part concerning the H\"older scale the full strength of our theorem is required. We compare our approach with that in \cite{Cieszynski2026} at the end of the introduction.\\
    (d) Our counterexample is weaker than Kahane's \cite[Proposition~P4]{Kahane2005} in the special case $\sigma_{n,\varepsilon} = (n\varepsilon)^{-1} \sin(n\varepsilon)$, since we do not know whether $f$ can be chosen to have degree zero and since we cannot prescribe the limit on the left side of \eqref{eq:main} (or even prove that this limit exists).\\
    (e) We emphasize that $f$ depends on the family $(\sigma_{n,\varepsilon})$. In \cite{Cieszynski2026} it is observed that, under the literal definition of a summation process, every $f\in C(\T;\Sone)$ (indeed, every $f\in VMO(\T;\Sone)$) admits an $f$-dependent summation process recovering $\degmap f$. Thus \cite[Open Problem~5.7]{Brezis2023} has an affirmative answer and \cite[Open Problem~5.8]{Brezis2023} a negative answer. Since the construction in \cite{Cieszynski2026} may depend explicitly on $\degmap f$, it is natural to seek a meaningful variant imposing additional structural restrictions on the summation processes.
\end{remark}


\subsection*{Historical remarks}

In \cite{BrezisNirenberg1995,Brezis1997}, the question was raised whether a standard
summation process can be used to regularize the formula
\eqref{eq:smooth-degree} for $f\in C(\T;\Sone)$. Korevaar
\cite{Korevaar1999} proved that, for continuous circle maps, the symmetric
sums $\sum_{|n|\leq N}n|\wh f(n)|^2$ may diverge and may also converge to a value different from $\degmap f$;
in fact, his convergent examples allow any value different from the
degree. He also gave examples for which the corresponding series is not
Abel--Poisson summable. Thus neither symmetric summation nor
Abel--Poisson summation recovers the degree for all continuous circle
maps. Korevaar's work was followed by work of Kahane
\cite{Kahane2005} and Brezis \cite{Brezis2006}, which we have already mentioned.

A closely related question is whether the modulus data $\{|\wh f(n)|\}_{n\in\Z}$ determine the degree. That this is the case for $f\in W^{1/3,3}(\T;\Sone)$ is a consequence of \eqref{eq:sinc-formula}. Conversely, Bourgain and Kozma \cite{BourgainKozma2007} have constructed $f,g\in C(\T;\Sone)$ with
\[
 |\wh f(n)|=|\wh g(n)|\quad\text{for all }n\in\Z,
 \qquad
 \degmap f\ne\degmap g \,.
\]
Whether such a pair can be constructed with both maps in $W^{1/p,p}$ for
some $p>3$, or in $C^{0,\alpha}$ for some $0<\alpha\leq1/3$, remains an
open question; see \cite[Open Problem~5.5]{Brezis2023}.

We emphasize that the Bourgain--Kozma construction shows that every summation process fails on some continuous circle map; see \cite[Corollary~5.1]{Brezis2023}. The thrust of our theorem is to show that the counterexample can be chosen in the critical H\"older space.


\subsection*{Structure of our proof}

The proof of Theorem~\ref{thm:main} has two main ingredients.  

$\bullet$ First, we prove a critical barycentric identity: every sufficiently small ball in the real $C^{0,1/3}$ space of phase functions supports an absolutely summable signed family of degree-zero maps whose averaged autocorrelation asymmetry is exactly a nonzero multiple of $\sin t$.  The latter identity is obtained by combining an autocorrelation synthesis theorem for small smooth phases with Kahane's lacunary solution of a cubic functional equation.

$\bullet$ Second, Baire category and translation averaging reduce a hypothetical universal process to a sequence of bounded odd multipliers which converges to the identity at every fixed frequency and is uniformly bounded on a full ball of small degree-zero perturbations.  

\medskip

We present the first ingredient in Sections \ref{sec:synthesis} and \ref{sec:kahane}. Theorem \ref{thm:synthesis} is the main result of Section \ref{sec:synthesis} and concerns the synthesis by autocorrelations of small phases. It is combined in Section \ref{sec:kahane} with Kahane's counterexample to prove the critical barycentric identity in Proposition \ref{prop:critical-barycenter}. Its Corollary \ref{cor:barycenter-multiplier} is the culmination of these two sections.

The second ingredient is presented in Section \ref{sec:main-proof}, where we complete the proof of Theorem \ref{thm:main}.

\medskip

While completing this manuscript, we learned of the recent work of
Cieszyński \cite{Cieszynski2026}, who proved the corresponding
nonexistence result for $f\in C^{0,\alpha}(\T;\Sone)$ with $\alpha<1/3$. This also gives the negative answer for $W^{1/p,p}$ for every $p>3$, while leaving the endpoint $C^{0,1/3}$ open. The argument in \cite{Cieszynski2026} is based on the Baire category theorem, as is ours, but in contrast to ours it uses multiplier-adapted degree-zero
perturbations built from quotients of Blaschke factors and relies crucially on a gain derived from the strict inequality $\alpha<1/3$. At the endpoint, we replace this unavailable gain by a multiplier-independent exact critical
barycentric identity (Proposition \ref{prop:critical-barycenter} and Corollary \ref{cor:barycenter-multiplier}), obtained from small-phase autocorrelation synthesis (Theorem \ref{thm:synthesis}) and Kahane's lacunary cubic identity (Lemma \ref{lem:kahane}).



\section{Synthesis by autocorrelations of small phases}\label{sec:synthesis}

The purpose of this section is to prove the following synthesis theorem for odd periodic functions with zero derivative at the origin in terms of autocorrelation functions. The imaginary part of the autocorrelation function of $f\in L^2(\T)$ is denoted by
\begin{equation}
    \label{eq:imag-autocorrelation}
    B_f(t) :=\operatorname{Im}\int_\T f(x+t)\overline{f(x)}\,\frac{dx}{2\pi} \qquad\text{for}\ t\in\T \,.
\end{equation}

\begin{theorem}[Small-phase autocorrelation synthesis]\label{thm:synthesis}
Let $P$ be a smooth real odd $2\pi$-periodic function with $P'(0)=0$.  Given $\delta>0$, there are real coefficients $(\Lambda_\ell)_{\ell\geq1}$ and smooth real periodic functions $(\theta_\ell)_{\ell\geq1}$ such that
\begin{align}
 &\sum_{\ell\geq1}|\Lambda_\ell|<\infty,\label{eq:synthesis-l1}\\
 &\|\theta_\ell\|_{C^1}<\delta
 \quad\text{for every }\ell,\label{eq:synthesis-small}\\
 &P(t)=\sum_{\ell\geq1}\Lambda_\ell
 B_{e^{i\theta_\ell}}(t)
 \quad\text{for every }t\in\T.\label{eq:synthesis-identity}
\end{align}
The series in \eqref{eq:synthesis-identity} converges absolutely and uniformly.
\end{theorem}

The proof shows that the smoothness assumption on $P$ can be relaxed to $P\in Y_0$, where $Y_0$ is introduced in the next subsection. We will only apply this theorem in the smooth situation.

This whole section is dedicated to the proof of this theorem.


\subsection{A weighted sine space}

Let $Y$ be the Banach space of real odd functions
\begin{equation}\label{eq:Y-definition}
 H(t)=\sum_{n\geq1}h_n\sin(nt)
 \quad\text{such that}\quad
 \|H\|_Y:=\sum_{n\geq1}n^3|h_n|<\infty.
\end{equation}
The series and its first derivative converge absolutely and uniformly.  Define the closed subspace
\begin{equation}\label{eq:Y0-definition}
 Y_0:=\{H\in Y:H'(0)=0\}
 =\left\{H\in Y:\sum_{n\geq1}nh_n=0
 \right\}.
\end{equation}
For positive integers $r,s$, set
\begin{equation}\label{eq:Ers}
 E_{r,s}(t):=\sin(rt)+\sin(st)-\sin((r+s)t)
 \qquad\text{for all}\ t\in\T \,.
\end{equation}
We note that $E_{r,s}\in Y_0$.

\begin{lemma}[Balanced atomic decomposition]\label{lem:balanced-decomp}
There are a countable index set $\Gamma$ and, for each $\gamma\in\Gamma$, positive integers $r_\gamma$ and $s_\gamma$ satisfying
\[
 r_\gamma\leq s_\gamma\leq r_\gamma+1 \,,
\]
and a linear map $Y_0\ni H\mapsto (a_\gamma(H))_{\gamma\in\Gamma}$ such that, for all $H\in Y_0$,
\begin{equation}\label{eq:atomic-decomposition}
 H=\sum_{\gamma\in\Gamma} a_\gamma(H)E_{r_\gamma,s_\gamma}
\end{equation}
with convergence in $Y$ and, with $N_\gamma:=r_\gamma+s_\gamma$, 
\begin{equation}\label{eq:atomic-bound}
 \sum_{\gamma\in\Gamma} |a_\gamma(H)|N_\gamma^3
 \leq\frac32\|H\|_Y.
\end{equation}
\end{lemma}

In what follows we will call a pair $(r,s)$ of positive integers \emph{balanced} if $r\leq s\leq r+1$. For example, the sum in \eqref{eq:atomic-decomposition} is over balanced pairs.

\begin{proof}
\emph{Step 1.} We expand $H$ in its sine series, $H(t)=\sum_{n\geq1}h_n\sin(nt)$. Since $H\in Y_0$, we have
\[
 h_1=-\sum_{n\geq2}nh_n \,.
\]
This allows us to write
\begin{equation}\label{eq:H-by-G}
 H=\sum_{n\geq2}h_nG_n \,,
\end{equation}
where, for $n\geq1$, we have set
\[
 G_n(t):=\sin(nt)-n\sin t \,.
\]
Note in particular, that $G_1=0$. In view of \eqref{eq:H-by-G} it suffices to decompose each $G_n$ into atoms $E_{r,s}$ with balanced pairs $(r,s)$.

\medskip

\emph{Step 2.} We shall show that for every $n$ there is a finite index set $\mathcal T_n$ such that
\begin{equation}\label{eq:G-tree}
    G_n=-\sum_{v\in\mathcal T_n}E_{r_v,s_v} \,.
\end{equation}

The basic ingredient in the proof of \eqref{eq:G-tree} is the elementary identity
\begin{equation}\label{eq:G-recursion}
 G_n=G_r+G_s-E_{r,s}
 \qquad\text{where}\ n= r+s \,.
\end{equation}

Before proving \eqref{eq:G-tree} for general $n$, we treat some basic examples. For $n=2$, we have, using $G_1=0$,
$$
G_2 = G_1 + G_1 - E_{1,1} = - E_{1,1} \,.
$$
Thus, $\mathcal T_2$ consists of a single element and $(r_v,s_v)=(1,1)$ for this element $v$. Next, for $n=3$, we have, using the formulas for $G_1$ and $G_2$, 
$$
G_3 = G_1 + G_2 - E_{1,2} = - E_{1,1} - E_{1,2} \,.
$$
Thus, $\mathcal T_3$ consists of two elements, and $(r_v,s_v)$ takes the two values $(1,1)$ and $(1,2)$. Next, for $n=4$, we have, using the formula for $G_2$,
$$
G_4 = G_2 + G_2 - E_{2,2} = - E_{1,1} - E_{1,1} - E_{2,2} \,.
$$
We emphasize that $\mathcal T_4$ consists of \emph{three} elements and that the pair $(1,1)$ appears \emph{twice}, while $(2,2)$ appears once.

We now generalize this procedure to arbitrary integers $n\geq 2$. We recursively split every integer $N\geq2$ into
\[
 r=\lfloor N/2\rfloor,
 \qquad
 s=\lceil N/2\rceil,
\]
and stop at the leaves labeled $1$.  Denote by $\mathcal T_n$ the set of internal vertices of this finite binary tree, with the two child subtrees regarded as disjoint copies when they have the same label.  Iterating \eqref{eq:G-recursion} and using $G_1=0$ yields the claimed decomposition \eqref{eq:G-tree}.

\medskip

\emph{Step 3.}
Denoting by $N_v=r_v+s_v$ the label at an internal vertex $v\in\mathcal T_n$, we claim that
\begin{equation}\label{eq:tree-cost}
 C_n:=\sum_{v\in\mathcal T_n}N_v^3 \leq\frac32n^3 \,.
\end{equation}
The assertion is trivial for $n=1$.  For $n\geq2$, with the balanced split $n=r+s$,
\[
 C_n=n^3+C_r+C_s.
\]
By induction,
\[
 C_n\leq n^3+\frac32(r^3+s^3).
\]
For a balanced pair, $r^3+s^3\leq n^3/3$.  When $n$ is even this is immediate; when $n\geq3$ is odd,
\[
 r^3+s^3=\frac{n^3+3n}{4}\leq\frac{n^3}{3}.
\]
This proves \eqref{eq:tree-cost}.

\medskip

\emph{Step 4.}
In Steps 1 and 2 we have shown that $H$ has a decomposition of the form \eqref{eq:atomic-decomposition} with 
$$
\Gamma:=\{ (n,v) :\ n\geq 2 \,,\ v\in\mathcal T_n \}
$$
and with coefficients $a_{(n,v)}(H)= - h_n$ depending linearly on $H$.

Now we use the bound from Step 3 to prove \eqref{eq:atomic-bound} and convergence in $Y$. Indeed, \eqref{eq:tree-cost} gives
\[
 \sum_\gamma|a_\gamma(H)|N_\gamma^3
 \leq\sum_{n\geq2}|h_n|C_n
 \leq\frac32\sum_{n\geq2}n^3|h_n| \leq \frac32 \|H\|_Y \,.
\]
Finally,
\begin{equation}
    \label{eq:normers}
    \|E_{r,s}\|_Y=r^3+s^3+(r+s)^3\leq2(r+s)^3,
\end{equation}
with the same bound when $r=s$ after combining the two equal sine terms.  Hence the expansion converges absolutely in $Y$.
\end{proof}

Different indices $\gamma$ may correspond to the same atom $E_{r_\gamma,s_\gamma}$, so representations in this redundantly indexed family are not intrinsically unique. For instance, the two distinct child vertices in the tree for $G_4$ both correspond to $E_{1,1}$. The construction above nevertheless fixes a particular bounded linear coefficient-selection map, satisfying \eqref{eq:atomic-bound}, and this fixed choice is used below.


\subsection{Exact cubic atoms}

For every balanced pair $(r,s)$, we define
\begin{equation}\label{eq:phi-rs}
 \phi_{r,s}(x):=
 \begin{cases}
 \displaystyle
 \frac{\cos(rx)}{r}+\frac{\cos(sx)}{s}
 +\frac{\sin((r+s)x)}{r+s},&r<s,\\[2.2ex]
 \displaystyle
 \frac{\cos(rx)}{r}+\frac{\sin(2rx)}{2r},&r=s.
 \end{cases}
\end{equation}

\begin{lemma}[Cubic autocorrelation atoms]\label{lem:cubic-atoms}
For every balanced pair $(r,s)$,
\begin{equation}\label{eq:cubic-atom}
 \int_\T
 \big(\phi_{r,s}(x+t)-\phi_{r,s}(x)\big)^3\,\frac{dx}{2\pi}
 =\lambda_{r,s}E_{r,s}(t),
\end{equation}
where
\begin{equation}\label{eq:lambda-rs}
 \lambda_{r,s}:=
 \begin{cases}
 \displaystyle-\frac{3}{rs(r+s)},&r<s,\\[1.5ex]
 \displaystyle-\frac{3}{4r^3},&r=s.
 \end{cases}
\end{equation}
Moreover, for $0\leq k\leq3$, we have, setting $N:=r+s$,
\begin{equation}\label{eq:phi-derivative-bound}
 \|\phi_{r,s}^{(k)}\|_\infty\leq C N^{k-1}.
\end{equation}
In particular,
\begin{equation}\label{eq:lambda-comparable}
 cN^{-3}\leq|\lambda_{r,s}|\leq CN^{-3}.
\end{equation}
\end{lemma}

\begin{proof}
The derivative bounds follow immediately from \eqref{eq:phi-rs} and the balance relations $r,s\simeq N$.

We prove \eqref{eq:cubic-atom} by Fourier expansion.  If
\[
 \phi(x)=\sum_{k\in\Z}\phi_k e^{ikx},
\]
then
\begin{equation}\label{eq:cubic-fourier-general}
 \int_\T(\phi(x+t)-\phi(x))^3\,\frac{dx}{2\pi}
 =\sum_{k_1+k_2+k_3=0}
 \phi_{k_1}\phi_{k_2}\phi_{k_3}
 \prod_{j=1}^3(e^{ik_jt}-1).
\end{equation}
For positive $r,s$ and $N=r+s$, one has
\begin{equation}\label{eq:triangle-product}
 (e^{irt}-1)(e^{ist}-1)(e^{-iNt}-1)
 =2iE_{r,s}(t).
\end{equation}

Suppose first that $r<s$.  The relevant Fourier coefficients are
\[
 \phi_{\pm r}=\frac1{2r},
 \qquad
 \phi_{\pm s}=\frac1{2s},
 \qquad
 \phi_N=-\frac{i}{2N},
 \qquad
 \phi_{-N}=\frac{i}{2N}.
\]
The only nonvanishing resonant contribution to \eqref{eq:cubic-fourier-general} comes from permutations of $(r,s,-N)$ and $(-r,-s,N)$.  There is one apparent exception: when $(r,s)=(1,2)$, the triples $(r,r,-s)$ and $(-r,-r,s)$ are also resonant.  Their common Fourier coefficient is real, while
\[
 (e^{irt}-1)^2(e^{-i2rt}-1)=2iE_{r,r}(t),
 \qquad
 (e^{-irt}-1)^2(e^{i2rt}-1)=-2iE_{r,r}(t).
\]
Thus these two contributions cancel exactly.  Using the six permutations of each of the two main resonances and \eqref{eq:triangle-product}, we obtain
\begin{align*}
 &6\left(\frac1{2r}\right)
 \left(\frac1{2s}\right)
 \left(\frac{i}{2N}\right)(2iE_{r,s})\\
 &\qquad+
 6\left(\frac1{2r}\right)
 \left(\frac1{2s}\right)
 \left(-\frac{i}{2N}\right)(-2iE_{r,s})
 =-\frac{3}{rsN}E_{r,s}.
\end{align*}

If $r=s$, the only resonance is $(r,r,-2r)$ and its conjugate.  There are three permutations of each.  Since
\[
 \phi_{\pm r}=\frac1{2r},
 \qquad
 \phi_{2r}=-\frac{i}{4r},
 \qquad
 \phi_{-2r}=\frac{i}{4r},
\]
formula \eqref{eq:triangle-product} gives
\[
 -\frac{3}{4r^3}E_{r,r}(t).
\]
This proves \eqref{eq:cubic-atom}.  The comparability \eqref{eq:lambda-comparable} follows from balance.
\end{proof}


\subsection{A uniform fifth-order estimate}

For a balanced pair $(r,s)$, we set
$$
b_{r,s}(\tau,t) := B_{e^{i\tau\phi_{r,s}}}(t) = \operatorname{Im}\int_\T
 e^{i\tau(\phi_{r,s}(x+t)-\phi_{r,s}(x))}\,\frac{dx}{2\pi}.
$$

\begin{lemma}[Fifth-order bound]\label{lem:fifth-order}
There exists a universal constant $C$ such that for every balanced pair $(r,s)$ and every $|\tau|\leq1$, we have, with $N:=r+s$,
\begin{equation}\label{eq:fifth-order-bound}
 \left\|\frac{\partial^5}{\partial\tau^5}
 b_{r,s}(\tau)\right\|_Y
 \leq C N^{-2}.
\end{equation}
\end{lemma}

In the proof we will use the following observation concerning the imaginary part of the autocorrelation function. By Parseval's identity, we have, for each $h\in L^2(\T)$,
    $$
    \int_\T h(x+t) \overline{h(x)}\,\frac{dx}{2\pi} = \sum_{n\in\Z} |\wh h(n)|^2 e^{int} \,.
    $$
    The series on the right side is absolutely and uniformly convergent. Taking the imaginary part, we find
    \begin{equation}\label{eq:B-sine-series}
        B_h(t) = \sum_{n\in\Z} |\wh h(n)|^2 \sin(nt) = \sum_{n\geq 1} \left( |\wh h(n)|^2 - |\wh h(-n)|^2 \right) \sin(nt) \,.
    \end{equation}

\begin{proof}
We abbreviate $\phi:=\phi_{r,s}$ and $f_\tau:=e^{i\tau\phi}$.  Put
\[
 u_a:=\phi^a e^{i\tau\phi},
 \qquad 0\leq a\leq5.
\]
We first justify differentiation in the weighted Fourier space used below. For
smooth functions $u,v$, Cauchy--Schwarz gives
\begin{equation}\label{eq:weighted-bilinear}
 \sum_{n\in\Z}|n|^3
 |\wh u(n)\overline{\wh v(n)}|
 \leq \|u\|_{\dot H^3}\|v\|_2.
\end{equation}
Moreover, for every $0\leq a\leq5$, the map
$\tau\mapsto \phi^a e^{i\tau\phi}$ is smooth with values in both $H^3$ and
$L^2$. Thus \eqref{eq:weighted-bilinear} and the Banach-valued product rule
show that the Fourier-energy sequence
$\tau\mapsto (|\wh{f_\tau}(n)|^2)_{n\in\Z}$ is five times continuously
differentiable with values in
$\ell^1(\Z\setminus\{0\};|n|^3)$. Differentiation gives
\begin{equation}\label{eq:fifth-energy-derivative}
    \frac{d^5}{d\tau^5} \, |\wh{f_\tau}(n)|^2
 =\sum_{a=0}^5\binom5a i^a(-i)^{5-a}
 \wh u_a(n)\,\overline{\wh u_{5-a}(n)} \,.
\end{equation}
It follows from \eqref{eq:B-sine-series} with $h=f_\tau$ that the sine coefficient of $b_{r,s}(\tau)$ at frequency $n\geq1$ is $|\wh f_\tau(n)|^2 - |\wh f_\tau(-n)|^2$.  Hence
\begin{align}
 \left\|\frac{\partial^5}{\partial\tau^5}
 b_{r,s}(\tau)\right\|_Y
 &\leq\sum_{n\in\Z}|n|^3 \left|\frac{d^5}{d\tau^5} \, |\wh{f_\tau}(n)|^2 \right| \notag\\
 &\lesssim \sum_{a=0}^5
 \left(\sum_n|n|^6|\wh u_a(n)|^2\right)^{1/2}
 \left(\sum_n|\wh u_{5-a}(n)|^2\right)^{1/2}\notag\\
 &=\sum_{a=0}^5
 \|u_a\|_{\dot H^3}\,\|u_{5-a}\|_2 \,.
 \label{eq:Y-by-Sobolev}
\end{align}
We used Cauchy--Schwarz in the second line.

The derivative bounds \eqref{eq:phi-derivative-bound}, the product rule, and $|\tau|\leq1$ imply
\begin{equation}\label{eq:ua-estimates}
 \|u_a\|_{H^3}\lesssim N^{3-a},
 \qquad
 \|u_a\|_2\lesssim N^{-a},
 \qquad \text{for all}\ 0\leq a\leq5.
\end{equation}
For completeness, we prove these bounds.  For $0\leq j\leq3$, every term obtained by applying $j$ derivatives to $\phi^a e^{i\tau\phi}$ is a constant multiple of
\[
 \phi^{a-q}
 \prod_{\nu=1}^{q}\phi^{(k_\nu)}
 \prod_{\mu=1}^{v}\phi^{(\ell_\mu)}
 e^{i\tau\phi},
\]
where $0\leq q\leq a$, all $k_\nu,\ell_\mu\geq1$, and
\[
 \sum_{\nu=1}^{q}k_\nu+
 \sum_{\mu=1}^{v}\ell_\mu=j.
\]
The first product records derivatives falling on $\phi^a$, while the second records those generated by differentiating the exponential.  By \eqref{eq:phi-derivative-bound}, its absolute value is at most a constant times
\[
 N^{-(a-q)}
 N^{\sum_\nu(k_\nu-1)}
 N^{\sum_\mu(\ell_\mu-1)}
 = N^{j-a-v} \leq N^{j-a} \,.
\]
Taking $j=0,1,2,3$ proves \eqref{eq:ua-estimates}.  

Thus every product in \eqref{eq:Y-by-Sobolev} is bounded by a constant times 
\[
 N^{3-a}N^{-(5-a)}= N^{-2} \,.
\]
This proves \eqref{eq:fifth-order-bound}.
\end{proof}

\begin{lemma}[Approximate triangle atoms]\label{lem:approx-atoms}
For $0<\varepsilon\leq1/2$, define
\begin{equation}\label{eq:Etilde}
 \widetilde E_{r,s}^{\varepsilon}
 :=\frac{2B_{e^{i\varepsilon\phi_{r,s}}}
 -B_{e^{2i\varepsilon\phi_{r,s}}}}
 {\varepsilon^3\lambda_{r,s}}.
\end{equation}
Then $\widetilde E_{r,s}^{\varepsilon}\in Y_0$ and
\begin{equation}\label{eq:Etilde-error}
 \|\widetilde E_{r,s}^{\varepsilon}-E_{r,s}\|_Y
 \leq C\varepsilon^2(r+s).
\end{equation}
\end{lemma}

\begin{proof}
The map $\tau\mapsto b_{r,s}(\tau)$ is odd because
\[
 B_{e^{-i\tau\phi}}=-B_{e^{i\tau\phi}}.
\]
Consequently,
\[
 b_{r,s}(0)=b_{r,s}^{(2)}(0)=b_{r,s}^{(4)}(0)=0.
\]
Furthermore,
\[
 b_{r,s}'(0,t)
 =\int_\T(\phi(x+t)-\phi(x))\,\frac{dx}{2\pi}=0,
\]
and Lemma~\ref{lem:cubic-atoms} gives
\[
 b_{r,s}^{(3)}(0)=-\lambda_{r,s}E_{r,s}.
\]
The Banach-valued Taylor formula and Lemma~\ref{lem:fifth-order} therefore yield
\begin{equation}\label{eq:b-taylor}
 \left\|b_{r,s}(\tau)
 +\frac{\tau^3\lambda_{r,s}}6E_{r,s}\right\|_Y
 \leq C|\tau|^5N^{-2}
 \qquad(|\tau|\leq1).
\end{equation}
The cubic terms in $2b(\varepsilon)-b(2\varepsilon)$ combine to $\varepsilon^3\lambda_{r,s}E_{r,s}$.  Applying \eqref{eq:b-taylor} at $\varepsilon$ and $2\varepsilon$, dividing by $\varepsilon^3|\lambda_{r,s}|$, and using \eqref{eq:lambda-comparable} gives \eqref{eq:Etilde-error}.

Finally, if $f=e^{i\theta}$ with a smooth periodic real phase function, then
\[
 B_f'(0)=\operatorname{Im}\int_\T f'(x)\overline{f(x)}\,\frac{dx}{2\pi}
 =\int_\T\theta'(x)\,\frac{dx}{2\pi}=0.
\]
Moreover, $f$ and $B_f$ are smooth, so the sine coefficients of $B_f$ have
the weighted summability required in the definition of $Y$. Thus each term
in \eqref{eq:Etilde} lies in $Y_0$.
\end{proof}


\subsection{Proof of the small-phase autocorrelation synthesis theorem}

We are now in a position to prove the main result of this section.

\begin{proof}[Proof of Theorem \ref{thm:synthesis}]
\emph{Step 1.}
According to Lemma~\ref{lem:balanced-decomp}, we have a bounded linear coefficient-selection map $Y_0\ni H \mapsto (a_\gamma(H))$. In terms of this map, for $0<\varepsilon\leq1/2$ and $H\in Y_0$, set
\begin{equation}\label{eq:S-epsilon}
 S_\varepsilon H
 :=\sum_{\gamma\in\Gamma} a_\gamma(H)
 \widetilde E_{r_\gamma,s_\gamma}^{\varepsilon}.
\end{equation}
This series converges absolutely in $Y$. Indeed, by \eqref{eq:normers} and \eqref{eq:Etilde-error},
\[
 \|\widetilde E_{r,s}^{\varepsilon}\|_Y
 \leq \|E_{r,s}\|_Y+C\varepsilon^2(r+s)
 \leq C(r+s)^3,
\]
and therefore \eqref{eq:atomic-bound} implies
\[
 \sum_\gamma |a_\gamma(H)|
 \|\widetilde E_{r_\gamma,s_\gamma}^{\varepsilon}\|_Y
 <\infty.
\]
By Lemma \ref{lem:approx-atoms}, we have $\widetilde E_{r,s}^{\varepsilon}\in Y_0$, so the closedness of $Y_0$ implies that $S_\varepsilon H\in Y_0$. To summarize, $S_\varepsilon$ is a bounded linear operator from $Y_0$ to $Y_0$.

By \eqref{eq:Etilde-error} and \eqref{eq:atomic-bound},
\begin{align*}
 \|(S_\varepsilon-I)H\|_Y
 &\leq C\varepsilon^2
 \sum_\gamma|a_\gamma(H)|N_\gamma\\
 &\leq C\varepsilon^2
 \sum_\gamma|a_\gamma(H)|N_\gamma^3
 \leq C\varepsilon^2\|H\|_Y.
\end{align*}
Choose $0<\varepsilon\leq1/2$ so small that
\begin{equation}\label{eq:Neumann-small}
 \|S_\varepsilon-I\|_{Y_0\to Y_0} \leq \frac12
\end{equation}
and, in addition, $2C_0\varepsilon<\delta$, where $C_0:= \sup_{(r,s) \,\mathrm{balanced}} \|\phi_{r,s}\|_{C^1}$, which is finite by Lemma~\ref{lem:cubic-atoms}.  Then $S_\varepsilon$ is invertible by the Neumann series.  

\medskip

\emph{Step 2.}
Let $P$ be a smooth, real, odd, $2\pi$-periodic function with $P'(0)=0$. Then $P\in Y_0$. Put
\[
 H:=S_\varepsilon^{-1}P \,.
\]
By \eqref{eq:S-epsilon} and \eqref{eq:Etilde},
\begin{equation}\label{eq:P-expanded-B}
 P
 =\sum_\gamma
 \frac{a_\gamma(H)}{\varepsilon^3\lambda_{r_\gamma,s_\gamma}}
 \left(2B_{e^{i\varepsilon\phi_{r_\gamma,s_\gamma}}}
 -B_{e^{2i\varepsilon\phi_{r_\gamma,s_\gamma}}}
 \right).
\end{equation}
The scalar coefficients in this expansion are summable.  Indeed, by \eqref{eq:lambda-comparable} and \eqref{eq:atomic-bound},
\begin{equation}\label{eq:coefficient-total-variation}
 \sum_\gamma
 \frac{3|a_\gamma(H)|}{\varepsilon^3|\lambda_{r_\gamma,s_\gamma}|}
 \leq C\varepsilon^{-3}
 \sum_\gamma|a_\gamma(H)|N_\gamma^3
 <\infty.
\end{equation}
Since $|B_f(t)|\leq1$, expansion \eqref{eq:P-expanded-B} therefore converges absolutely and uniformly after its two terms are enumerated as a single sequence.  Regrouping the two consecutive terms associated with each atom recovers the $Y$-convergent grouped series defining $S_\varepsilon H=P$. Since the split series converges unconditionally and uniformly, its uniform sum is therefore exactly $P$.  Finally, the phases in \eqref{eq:P-expanded-B} are $\varepsilon\phi_{r,s}$ and $2\varepsilon\phi_{r,s}$; their $C^1$ norms are smaller than $\delta$ by our choice of $\varepsilon$.
\end{proof}


\section{Kahane's cubic function and the critical barycenter}\label{sec:kahane}

The purpose of this section is to prove the following proposition. We recall that the notation $B_f$ was introduced in \eqref{eq:imag-autocorrelation}.

\begin{proposition}[Critical barycenter]\label{prop:critical-barycenter}
For every $r>0$, there are real numbers $(\Lambda_\ell)_{\ell\geq1}$ with
\begin{equation}\label{eq:Lambda-l1}
 \sum_{\ell\geq1}|\Lambda_\ell|<\infty \,,
\end{equation}
a family of real functions
\[
 \eta_{\ell,\tau}\in C^{0,1/3}(\T),
 \qquad \ell\geq1,\ \tau\in\T \,,
\]
such that, for every $\ell$, the map $(\tau,x)\mapsto\eta_{\ell,\tau}(x)$ is continuous, with
\begin{equation}\label{eq:small-barycenter-phases}
 \|\eta_{\ell,\tau}\|_{C^{0,1/3}}<r,
\end{equation}
and a constant $\kappa\ne0$ such that, for $h_{\ell,\tau}:=e^{i\eta_{\ell,\tau}}$, one has $\degmap h_{\ell,\tau}=0$ and
\begin{equation}\label{eq:critical-barycenter}
 \sum_{\ell\geq1}\Lambda_\ell
 \int_\T B_{h_{\ell,\tau}}(t)\,\frac{d\tau}{2\pi}
 =\kappa\sin t
 \qquad\text{for every }t\in\T.
\end{equation}
\end{proposition}

To prove this proposition, we will apply Theorem~\ref{thm:synthesis} with a particular function $P$. That function will be constructed in terms of a function that Kahane used in a related context \cite{Kahane2005}.

Before presenting the details of this argument, we deduce a simple consequence of it, which will be crucial in the proof of our main result.

For a bounded sequence $q:\Z\to\C$ and $f\in L^2(\T)$, set
\begin{equation}\label{eq:Qq}
 Q_q(f):=\sum_{n\in\Z}q(n)|\wh f(n)|^2.
\end{equation}
This series is absolutely convergent. If $q$ is real-valued and odd, then
$Q_q(f)$ is real and
\begin{equation}\label{eq:Qq-odd}
 Q_q(f)
 =\sum_{n\geq1}q(n)
 \big(|\wh f(n)|^2-|\wh f(-n)|^2\big).
\end{equation}

\begin{corollary}\label{cor:barycenter-multiplier}
Given $r>0$, let $(\Lambda_\ell)_{\ell\geq1}$, $(h_{\ell,\tau})_{\ell\geq 1, \tau\in\T}$ and $\kappa\neq 0$ be as in Proposition~\ref{prop:critical-barycenter}. Then every bounded real odd sequence $s$ satisfies
\begin{equation}\label{eq:barycenter-Q}
 \sum_{\ell\geq1}\Lambda_\ell
 \int_\T Q_s(h_{\ell,\tau})\,\frac{d\tau}{2\pi}
 =\kappa s(1).
\end{equation}
\end{corollary}

\begin{proof}
    Equation \eqref{eq:imag-autocorrelation} implies that, for each $h\in L^2(\T)$,
    $$
    2 \int_\T B_h(t) \sin(nt)\,\frac{dt}{2\pi} = |\wh h(n)|^2 - |\wh h(-n)|^2 
    \qquad\text{for all}\, n\geq1 \,.
    $$

    We apply this identity to $h=h_{\ell,\tau}$, integrate with respect to $\tau$, multiply by $\Lambda_\ell$ and sum over $\ell$ to obtain
    $$
    2 \int_\T \left( \sum_{\ell} \Lambda_\ell \int_\T B_{h_{\ell,\tau}}(t)\,\frac{d\tau}{2\pi} \right) \sin(nt)\,\frac{dt}{2\pi} = \sum_{\ell} \Lambda_\ell \int_\T \left( |\wh h_{\ell,\tau}(n)|^2 - |\wh h_{\ell,\tau}(-n)|^2 \right)\frac{d\tau}{2\pi}
    \quad\text{for all}\ n\geq 1\,.
    $$
    On the left side we interchanged the $t$-integration with the $\tau$-integration and the $\ell$-summation, which is justified since the series in \eqref{eq:critical-barycenter} converges uniformly because of $|B_{h_{\ell,\tau}}|\leq1$ and \eqref{eq:Lambda-l1}.

    Proposition \ref{prop:critical-barycenter} allows us to rewrite this as
    $$
    \kappa \mathbf 1_{\{n=1\}} = 2\kappa \int_\T \sin(t) \sin(nt)\,\frac{dt}{2\pi}
    = \sum_{\ell} \Lambda_\ell \int_\T \left( |\wh h_{\ell,\tau}(n)|^2 - |\wh h_{\ell,\tau}(-n)|^2 \right)\frac{d\tau}{2\pi}
    \qquad\text{for all}\ n\geq 1    \,.
    $$
    Multiplying by $s(n)$ and summing over $n\geq1$ gives \eqref{eq:barycenter-Q}. Here the interchange of the $n$-summation and the $\ell$-summation and the $\tau$-integration is justified by boundedness of $s$, \eqref{eq:Lambda-l1}, and $\sum_{n\in\Z}|\wh h_{\ell,\tau}(n)|^2=1$.
\end{proof}


\subsection{The critical cubic identity}

Let
\begin{equation}\label{eq:kahane-function}
 \psi(x):= \frac{2^{1/9}}{3^{1/3}} \,\sum_{j=0}^{\infty}2^{-j/3}\sin(2^j x) \qquad\text{for}\ x\in\T \,.
\end{equation}

The following lemma is due to Kahane
\cite{Kahane2005}. We use the corrected lower summation
index $j=0$: in \cite[Eq.~(13)]{Kahane2005} it is printed as $j=1$, while
\cite[Eq.~(17)]{Kahane2005} and the direct computation below show that
$j=0$ is required. We include a proof for completeness.

\begin{lemma}\label{lem:kahane}
The series \eqref{eq:kahane-function} defines a function in $C^{0,1/3}(\T;\R)$ and
\begin{equation}\label{eq:kahane-identity}
 \int_\T
 \big(\psi(x+t)-\psi(x)\big)^3\,\frac{dx}{2\pi}
 =\sin t
 \qquad\text{for all}\ t\in\T \,.
\end{equation}
\end{lemma}

\begin{proof}
Uniform convergence is immediate from $\sum_j2^{-j/3}<\infty$.  To prove the H\"older bound, fix $0<|h|\leq1$ and choose $m\geq0$ so that
\[
 2^{-m-1}<|h|\leq2^{-m}.
\]
For the low frequencies, we bound
\begin{align*}
 \sum_{j=0}^m2^{-j/3}
 |\sin(2^j(x+h))-\sin(2^j x)|
 &\leq |h|\sum_{j=0}^m2^{2j/3}\\
 &\lesssim |h|2^{2m/3}
 \lesssim |h|^{1/3}.
\end{align*}
For the high frequencies, we bound
\[
 \sum_{j>m}2^{-j/3}
 |\sin(2^j(x+h))-\sin(2^j x)|
 \leq2\sum_{j>m}2^{-j/3}
 \lesssim 2^{-m/3}
 \lesssim |h|^{1/3}.
\]
This proves the required estimate for $0<|h|\leq1$. For increments with
$|h|>1$, we have
\[
 |\psi(x+h)-\psi(x)|
 \leq2\|\psi\|_\infty
 \leq2\|\psi\|_\infty |h|^{1/3},
\]
so the estimate holds for all increments on $\T$. Thus
$\psi\in C^{0,1/3}$.

Let $A:= \frac{2^{1/9}}{3^{1/3}}$ and
\[
 \psi_L(x):=A\sum_{j=0}^L2^{-j/3}\sin(2^j x) \,.
\]
In the integral of $(\psi_L(x+t)-\psi_L(x))^3$, a triple of dyadic frequencies can sum to zero only through
\[
 2^j+2^j=2^{j+1},
 \qquad 0\leq j<L.
\]
Indeed, if one power of two is the sum of two others, uniqueness of binary expansion forces the two smaller powers to be equal.  We now compute the contribution of this resonance.  Put $k:=2^j$.  The elementary formula
\[
 \sin(k(x+t))-\sin(kx)
 =2\sin(kt/2)\cos(kx+kt/2)
\]
and the identity
\[
 \int_\T\cos^2(kx+a)\cos(2kx+2a)\,\frac{dx}{2\pi}=\frac14
\]
give
\begin{align*}
 &\int_\T
 \big(\sin(k(x+t))-\sin(kx)\big)^2
 \big(\sin(2k(x+t))-\sin(2kx)\big)\,\frac{dx}{2\pi}\\
 &\qquad=2\sin^2(kt/2)\sin(kt)
 =\sin(kt)-\frac12\sin(2kt).
\end{align*}
The multinomial coefficient is $3$, and the product of the three lacunary amplitudes is
\[
 A^3 2^{-2j/3}2^{-(j+1)/3}=A^3 2^{-j-1/3}.
\]
Summing over $j$ therefore gives
\begin{align}
 &\int_\T
 \big(\psi_L(x+t)-\psi_L(x)\big)^3\,\frac{dx}{2\pi}\notag\\
 &\qquad=\frac32A^3 2^{-1/3}
 \sum_{j=0}^{L-1}2^{-j}
 \big(2\sin(2^j t)-\sin(2^{j+1}t)\big).\label{eq:finite-kahane}
\end{align}
The sum telescopes:
\[
 \sum_{j=0}^{L-1}2^{-j}
 \big(2\sin(2^j t)-\sin(2^{j+1}t)\big)
 =2\big(\sin t-2^{-L}\sin(2^L t)\big).
\]
Since $3A^3 2^{-1/3}=1$, we obtain the exact finite identity
\begin{equation}\label{eq:finite-kahane-exact}
 \int_\T
 \big(\psi_L(x+t)-\psi_L(x)\big)^3\,\frac{dx}{2\pi}
 =\sin t-2^{-L}\sin(2^L t).
\end{equation}
The functions $\psi_L$ converge uniformly to $\psi$, so the cubic integrands converge uniformly in $x$ for each $t$; in fact the convergence is uniform in $(x,t)$.  Passing to the limit in \eqref{eq:finite-kahane-exact} proves \eqref{eq:kahane-identity}.
\end{proof}


\subsection{Proof of the critical barycenter proposition}

\begin{proof}[Proof of Proposition~\ref{prop:critical-barycenter}]
Fix $r>0$ and let $\psi$ be given by \eqref{eq:kahane-function}. We choose the constant $\rho>0$ so small that the function
\begin{equation}\label{eq:u-rho-psi}
 u:=\rho\psi,
\end{equation}
satisfies
\begin{equation}\label{eq:u-range-small}
 2\|u\|_\infty<\frac{\pi}{2}.
\end{equation}
Then, by Lemma~\ref{lem:kahane},
\begin{equation}\label{eq:u-cubic}
 \int_\T(u(x+t)-u(x))^3\,\frac{dx}{2\pi}
 =\rho^3\sin t.
\end{equation}

Let $\chi:\R\to\R$ be a smooth even function that equals $1$ on the interval
\[
 [-2\|u\|_\infty,2\|u\|_\infty]
\]
and has compact support in a compact subinterval of $(-\pi,\pi)$. We set
\begin{equation}\label{eq:P-cutoff}
 P(s):=\chi(s)s^3
 \qquad\text{for}\ -\pi<s<\pi
\end{equation}
and extend $P$ periodically to $\R$. In this way we obtain a smooth, real, odd function with $P'(0)=0$.

We apply Theorem~\ref{thm:synthesis} with a number $\delta>0$ to be fixed below and obtain the decomposition
\begin{equation}\label{eq:P-synthesis-barycenter}
 P(s)=\sum_{\ell\geq1}\Lambda_\ell B_{e^{i\theta_\ell}}(s),
\end{equation}
where
\begin{equation}\label{eq:barycenter-synthesis-bounds}
 \sum_\ell|\Lambda_\ell|<\infty,
 \qquad
 \|\theta_\ell\|_{C^1}<\delta.
\end{equation}
For $\ell\geq1$ and $\tau\in\T$, define
\begin{equation}\label{eq:h-composition}
 \eta_{\ell,\tau}(x):=\theta_\ell(u(x)+\tau)
 \qquad\text{for all}\ x\in\T \,.
\end{equation}
Here $\theta_\ell$ is viewed as a $2\pi$-periodic function on $\R$. Every function $e^{i\eta_{\ell,\tau}}$ has degree zero because $\eta_{\ell,\tau}$ is a real-valued periodic lift.  Moreover,
\begin{align}
 \|\eta_{\ell,\tau}\|_{C^{0,1/3}}
 &\leq\|\theta_\ell\|_\infty
 +\|\theta_\ell'\|_\infty[u]_{C^{0,1/3}}\notag\\
 &\leq\delta\big(1+[u]_{C^{0,1/3}}\big).\label{eq:composition-holder}
\end{align}
At this point we choose $\delta$ so small that the last expression is less than $r$.  This proves \eqref{eq:small-barycenter-phases}.

It remains to compute the averaged autocorrelation. Set $h_{\ell,\tau}:=e^{i\eta_{\ell,\tau}}$. For fixed $\ell$ and $t$, Fubini's theorem and a change of variables give
\begin{align}
 \int_\T B_{h_{\ell,\tau}}(t)\,\frac{d\tau}{2\pi}
 &=\operatorname{Im}\int_\T\int_\T
 e^{i\eta_{\ell,\tau}(x+t)} e^{-i\eta_{\ell,\tau}(x)}
 \,\frac{dx}{2\pi}\,\frac{d\tau}{2\pi}\notag\\
 &=\operatorname{Im}\int_\T\int_\T
 e^{i\theta_\ell(u(x+t)+\tau)}
 e^{-i\theta_\ell(u(x)+\tau)}
 \,\frac{dx}{2\pi}\,\frac{d\tau}{2\pi}\notag\\
 &=\int_\T B_{e^{i\theta_\ell}}(u(x+t)-u(x))\,\frac{dx}{2\pi}.
 \label{eq:composition-average}
\end{align}
Indeed, for each fixed $x$ one makes the change of variable $\tau'=\tau+u(x)$ in the inner integral.

We multiply \eqref{eq:composition-average} by $\Lambda_\ell$ and sum over $\ell$. On the right side we interchange the $x$-integration with the $\ell$-summation, which is justified by \eqref{eq:barycenter-synthesis-bounds} and $|B_{e^{i\theta_\ell}}|\leq1$. It follows from \eqref{eq:P-synthesis-barycenter}, \eqref{eq:u-range-small} and \eqref{eq:P-cutoff} that
$$
\sum_{\ell\geq 1} \Lambda_\ell B_{e^{i\theta_\ell}}(u(x+t)-u(x)) = P(u(x+t)-u(x)) = (u(x+t)-u(x))^3 \,.
$$
Thus, \eqref{eq:composition-average} yields
\begin{align*}
 \sum_\ell\Lambda_\ell
 \int_\T B_{h_{\ell,\tau}}(t)\,\frac{d\tau}{2\pi}
 &=\int_\T(u(x+t)-u(x))^3\,\frac{dx}{2\pi}
 =\rho^3\sin t,
\end{align*}
where the last equality follows from \eqref{eq:u-cubic}. Thus \eqref{eq:critical-barycenter} holds with $\kappa :=\rho^3\ne0$.
\end{proof}


\section{Proof of the main theorem}\label{sec:main-proof}

The purpose of this section is to prove Theorem \ref{thm:main}.

We assume, toward a contradiction, that there is a family $(\sigma_{n,\varepsilon})_{n\in\Z,\,0<\varepsilon<1}\subset\C$ satisfying \eqref{eq:process-bounded} and \eqref{eq:process-pointwise} such that for every $f\in C^{0,1/3}(\T;\Sone)$ we have
\[
 \sum_{n\in\Z}n\,|\wh f(n)|^2\sigma_{n,\varepsilon}
 \longrightarrow \degmap f
 \qquad\text{as }\varepsilon\downarrow0.
\]
We fix an arbitrary sequence $\varepsilon_j\downarrow0$ and put
\[
 q_j(n):=n\sigma_{n,\varepsilon_j}.
\]
Then \eqref{eq:process-bounded} and \eqref{eq:process-pointwise} imply that
\begin{align}
    \label{eq:process-bounded1}
    & \sup_{n\in\Z} |q_j(n)| < \infty
    && \text{for every}\ j \\
    \label{eq:q-pointwise}
    & q_j(n)\longrightarrow n
    && \text{for every}\ n\in\Z \,,
\end{align}
as well as
\begin{equation}\label{eq:q-recovers-degree}
 Q_{q_j}(f)\longrightarrow\degmap f
 \qquad\text{for every}\ f\in C^{0,1/3}(\T;\Sone) \,.
\end{equation}
We recall that $Q_q(f)$ was defined in \eqref{eq:Qq}.


\subsection{Reduction to real odd multipliers}

In this subsection we argue that, in addition to \eqref{eq:process-bounded1}, \eqref{eq:q-pointwise} and \eqref{eq:q-recovers-degree} we may assume that, for every $j$,
\begin{equation}\label{eq:q-odd}
    q_j \ \text{is real-valued and odd.}
\end{equation}

Indeed, given $q_j$ satisfying \eqref{eq:process-bounded1}, \eqref{eq:q-pointwise} and \eqref{eq:q-recovers-degree}, let
\begin{equation}\label{eq:odd-part}
 \widetilde q_j(n)
 :=\frac12\operatorname{Re}\big(q_j(n)-q_j(-n)\big) \,.
\end{equation}
Then clearly $\widetilde q_j$ still satisfies \eqref{eq:process-bounded1} and \eqref{eq:q-pointwise}, and it is real-valued and odd. Moreover, for $f\in C(\T;\Sone)$,
\[
 |\wh{\overline f}(n)| = |\wh f(-n)|
 \quad\text{for all}\ n\in\Z \,,
 \qquad
 \degmap(\overline f)=-\degmap f \,.
\]
Therefore, for all $f\in C^{0,1/3}(\T;\Sone)$,
\[
 Q_{\widetilde q_j}(f)
 =\frac12\operatorname{Re}
 \big(Q_{q_j}(f)-Q_{q_j}(\overline f)\big)
 \longrightarrow\degmap f.
\]
This shows that $\widetilde q_j$ still satisfies \eqref{eq:q-recovers-degree}.

This observation allows us to assume henceforth that $q_j$ satisfies \eqref{eq:process-bounded1}, \eqref{eq:q-pointwise}, \eqref{eq:q-recovers-degree} and \eqref{eq:q-odd}.


\subsection{A Baire-category ball}

Let
\[
 X:=C^{0,1/3}(\T;\R) \,,
\]
which is a real Banach space. If $\varphi\in X$, then $e^{i\varphi}\in C^{0,1/3}(\T;\Sone)$ and, since $e^{i\varphi}$ has the periodic lift $\varphi$, $\deg e^{i\varphi}=0$. Thus, \eqref{eq:q-recovers-degree} gives
\begin{equation}\label{eq:F-pointwise-zero}
 Q_{q_j}(e^{i\varphi})\longrightarrow0
 \qquad\text{for every }\varphi\in X.
\end{equation}

\begin{lemma}[Baire localization]\label{lem:baire-ball}
There are $\varphi_0\in X$, $r>0$, and $M<\infty$ such that for every $j\geq1$ and every $\eta\in X$ with $\|\eta\|_{C^{0,1/3}}<r$ we have
\begin{equation}\label{eq:baire-bound}
 |Q_{q_j}(e^{i(\varphi_0+\eta)})|\leq M \,.
\end{equation}
\end{lemma}

\begin{proof}
For each $N\in\N$, the set
\[
 E_N:=\{\varphi\in X: \ \sup_{j\geq1}|Q_{q_j}(e^{i\varphi})|\leq N\}
\]
is closed. Indeed, this follows from the continuity of $\varphi\mapsto e^{i\varphi}$ from $X$ to $L^2$ and from the continuity $f\mapsto Q_{q_j}(f)$ from $L^2$ to $\R$. More explicitly,
\[
 |Q_{q_j}(f)-Q_{q_j}(h)|
 \leq \|q_j\|_\infty
 \big(\|f\|_2+\|h\|_2\big)\|f-h\|_2.
\]

By \eqref{eq:F-pointwise-zero}, for every fixed $\varphi\in X$, the sequence $(Q_{q_j}(e^{i\varphi}))_j$ is bounded, so $X=\bigcup_{N\geq1}E_N$.  The Baire category theorem (see, e.g., \cite[Theorem 5.1]{VanNeerven2022}) implies that some $E_N$ has nonempty interior.  Any open ball contained in that interior gives the conclusion.
\end{proof}

With $\varphi_0\in X$ from Lemma \ref{lem:baire-ball}, we set
\begin{equation}\label{eq:g-center}
 g:=e^{i\varphi_0}.
\end{equation}
Then $g\in C^{0,1/3}(\T;\Sone)$ and $\degmap g=0$.


\subsection{Translation averaging}

For each $j$, we define a sequence $s_j:\Z\to\R$ by
\begin{equation}\label{eq:sj-definition}
 s_j(n):=\frac12\sum_{m\in\Z} |\wh g(m)|^2
 \big(q_j(m+n)-q_j(m-n)\big)
 \qquad\text{for all}\ n\in\Z\,.
\end{equation}
The sum converges absolutely because $q_j$ is bounded and $\sum_{m}|\wh g(m)|^2=1$.  It is immediate that $s_j$ is bounded, real, and odd.

\begin{lemma}[Translation-averaging identity]\label{lem:translation-averaging}
For every $h\in L^2(\T;\Sone)$,
\begin{equation}\label{eq:translation-average}
 Q_{s_j}(h)
 =\frac12\int_\T
 \big(Q_{q_j}(gh_\omega)-Q_{q_j}(g\overline{h_\omega})\big)\,\frac{d\omega}{2\pi} \,,
\end{equation}
where, for $\omega\in\T$,
\[
 h_\omega(x):=h(x+\omega),
 \qquad x\in\T.
\]
\end{lemma}

\begin{proof}
The map $\omega\mapsto h_\omega$ is strongly continuous in $L^2(\T)$.
Since multiplication by $g$ is an isometry on $L^2(\T)$ and
$Q_{q_j}$ is continuous on $L^2(\T)$, the integrands in
\eqref{eq:translation-average} are measurable.

For every $k\in\Z$,
\[
 \wh{(gh_\omega)}(k)
 =\sum_{n\in\Z}\wh g(k-n)\wh h(n)e^{in\omega},
\]
where the identity first holds for trigonometric polynomials. If $h_N\to h$
in $L^2(\T)$ are trigonometric polynomials, then the corresponding
expressions converge in $L^2(\T_\omega)$; hence the identity follows in
$L^2(\T_\omega)$ by approximation. Parseval in the $\omega$ variable gives
\begin{equation}\label{eq:averaged-energy-product}
 \int_\T | \wh {gh_\omega}(k)|^2 \,\frac{d\omega}{2\pi}
 =\sum_{n\in\Z} |\wh g(k-n)|^2 |\wh h(n)|^2 \,.
\end{equation}
Multiplying by $q_j(k)$ and summing in $k$ is justified by absolute
convergence, since
\[
 \sum_{k,n\in\Z}|q_j(k)|
 |\wh g(k-n)|^2|\wh h(n)|^2
 \leq\|q_j\|_\infty\|g\|_2^2\|h\|_2^2.
\]
Therefore, Fubini's theorem and the change of variables $m=k-n$ give
\begin{align*}
 \int_\T Q_{q_j}(gh_\omega)\,\frac{d\omega}{2\pi}
 &=\sum_{m,n\in\Z}q_j(m+n) |\wh g(m)|^2 |\wh h(n)|^2 \,.
\end{align*}
Arguing similarly and using $|\wh{\overline h}(k)| = |\wh h(-k)|$, we obtain
\[
 \int_\T Q_{q_j}(g\overline{h_\omega})\,\frac{d\omega}{2\pi}
 =\sum_{m,n\in\Z}q_j(m-n)|\wh g(m)|^2 |\wh h(n)|^2 \,.
\]
Subtracting and using \eqref{eq:sj-definition} proves \eqref{eq:translation-average}.
\end{proof}

\begin{proposition}\label{prop:s-properties}
Let $r$ and $M$ be as in Lemma \ref{lem:baire-ball}. The sequences $s_j$ have the following properties.
\begin{enumerate}[label=\textup{(\roman*)}]
\item If $h=e^{i\eta}$ with $\eta\in X$ and $\|\eta\|_{C^{0,1/3}}<r$, then
\begin{equation}\label{eq:s-uniform-bound}
 |Q_{s_j}(h)|\leq M
 \qquad\text{for every }j \,.
\end{equation}
\item For every fixed such $h$,
\begin{equation}\label{eq:s-pointwise-zero}
 \lim_{j\to\infty} Q_{s_j}(h)=0 \,.
\end{equation}
\item For every fixed $n\in\Z$,
\begin{equation}\label{eq:s-fixed-mode}
 \lim_{j\to\infty} s_j(n) = n \,.
\end{equation}
\end{enumerate}
\end{proposition}

\begin{proof}
Translations preserve the $C^{0,1/3}$ norm.  Hence for any $\omega\in\T$, the functions
\[
 \varphi_0+\eta(\cdot+\omega)
 \quad\text{and}\quad
 \varphi_0-\eta(\cdot+\omega)
\]
belong to the Baire ball in Lemma \ref{lem:baire-ball}.  The two terms on the right side of \eqref{eq:translation-average} are therefore bounded in absolute value by $M$, proving \eqref{eq:s-uniform-bound}.

For each fixed $\omega$, the maps $gh_\omega$ and $g\overline{h_\omega}$ have degree zero because all three factors admit periodic real lifts.  Thus \eqref{eq:q-recovers-degree} makes each term in the integrand of \eqref{eq:translation-average} converge to zero.  The bound from Lemma \ref{lem:baire-ball} supplies the integrable majorant $2M$, so dominated convergence proves \eqref{eq:s-pointwise-zero}.

Finally, Fourier translation in frequency gives
\begin{align*}
 Q_{q_j}(e^{inx}g)&=\sum_mq_j(m+n) \, |\wh g(m)|^2 \,,\\
 Q_{q_j}(e^{-inx}g)&=\sum_mq_j(m-n) \, |\wh g(m)|^2 \,.
\end{align*}
Consequently,
\begin{equation}\label{eq:s-degree-difference}
 s_j(n)=\frac12\big(Q_{q_j}(e^{inx}g)-Q_{q_j}(e^{-inx}g)\big).
\end{equation}
The two maps on the right have degrees $n$ and $-n$.  Applying \eqref{eq:q-recovers-degree} proves \eqref{eq:s-fixed-mode}.
\end{proof}


\subsection{Proof of Theorem~\ref{thm:main}}

Let us summarize what we have shown so far in this section: Under our contradiction assumption, we have produced constants $r>0$ and $M>0$, and bounded real odd sequences $s_j$ satisfying
\begin{align}
 |Q_{s_j}(e^{i\eta})|&\leq M
 &&\text{if }\|\eta\|_{C^{0,1/3}}<r,\label{eq:final-uniform}\\
 Q_{s_j}(e^{i\eta})&\longrightarrow0
 &&\text{for every fixed such }\eta,\label{eq:final-pointwise}\\
 s_j(1)&\longrightarrow1 \,.\label{eq:final-mode}
\end{align}
We apply Proposition~\ref{prop:critical-barycenter} with this $r$, and then  Corollary~\ref{cor:barycenter-multiplier} gives
\begin{equation}\label{eq:final-identity}
 \sum_{\ell\geq1}\Lambda_\ell
 \int_\T Q_{s_j}(h_{\ell,\tau})\,\frac{d\tau}{2\pi}
 =\kappa s_j(1)
\end{equation}
for some $\kappa\neq 0$. For every fixed $(\ell,\tau)$, the integrand on the left tends to zero by \eqref{eq:final-pointwise}. Moreover,
\[
 \big|\Lambda_\ell Q_{s_j}(h_{\ell,\tau})\big|
 \leq M|\Lambda_\ell|,
 \qquad
 \sum_{\ell\geq1}|\Lambda_\ell|<\infty.
\]
Dominated convergence on $\N\times\T$, equipped with counting measure times normalized Haar measure, therefore shows that the left side of \eqref{eq:final-identity} tends to zero. By \eqref{eq:final-mode}, the right side tends to $\kappa\ne0$, a contradiction.
\qed


\section*{Acknowledgments}
R.~L.~F. acknowledges partial support from US NSF grant DMS-1954995 and the DFG grants EXC-2111-390814868 and TRR 352-Project-ID 470903074.  P.~I. acknowledges partial support from the US NSF CAREER grant DMS-2152401, US NSF grant DMS-2554183, a Simons Fellowship, and a Humboldt Research Fellowship for Experienced Researchers.  The authors acknowledge the use of AI tools during the exploratory stage of this project.  All mathematical arguments and proofs in the final manuscript were checked and written by the authors.

\providecommand{\bysame}{\leavevmode\hbox to3em{\hrulefill}\thinspace}
\providecommand{\MR}{\relax\ifhmode\unskip\space\fi MR }
\providecommand{\MRhref}[2]{%
  \href{http://www.ams.org/mathscinet-getitem?mr=#1}{#2}
}
\providecommand{\href}[2]{#2}

\end{document}